\documentclass{amsart}

\makeatletter
 \def\LaTeX{\leavevmode L\raise.42ex
   \hbox{\kern-.3em\size{\sf@size}{0pt}\selectfont A}\kern-.15em\TeX}
\makeatother

\newcommand{\BibTeX}{{\rm B\kern-.05em{\sc
i\kern-.025emb}\kern-.08em\TeX}}
\newtheorem{corollary}{Corollary}[section]
\newtheorem{theorem}{Theorem}[section]

\newtheorem{remark}[theorem]{Remark}

\newtheorem{definition}[theorem]{Definition}

\numberwithin{equation}{section}

\begin{document}

\title[Boas-type formulas  in Banach  spaces]{Boas-type formulas   in Banach  spaces with applications to analysis on  manifolds }

\maketitle

\begin{center}
\title{\textit{Dedicated to 85th Birthday of my teacher Paul Butzer}}
\maketitle
\end{center}
\begin{center}

\author{Isaac Z. Pesenson }\footnote{ Department of  Mathematics, Temple University,
 Philadelphia,
PA 19122; pesenson@math.temple.edu. The author was supported in
part by the National Geospatial-Intelligence Agency University
Research Initiative (NURI), grant HM1582-08-1-0019. }

\end{center}

\bigskip



\begin{abstract}

The paper contains Boas-type formulas for trajectories of one-parameter groups of operators in Banach spaces. The results are illustrated using one-parameter groups of operators which appear in representations of Lie groups. 
\end{abstract}

\section{Preface}

My theachers were Vla\rm dimir Abramovich Rokhlin (my Master Thesis Advisor) and Selim Grigorievich Krein (my PhD Thesis Advisor). I first met Paul Butzer when I was about 50 years old but I also consider him as my teacher \rm since his work had an enormous influence on my carrier and ultimately on my life.

After I graduated from university it was almost impossible for me to go straight to graduate  school because of the Soviet discrimination towards Jews. However, I  was trying to do some mathematics on my own. One day I came across  a reference to the book by P. Butzer and H. Berens  "Semi-Groups of operators and  approximation", Springer, 1967.  Since I had some background in Lie groups and Lie semi-groups  and  knew nothing about approximation theory  the title sounded  very intriguing to me. Unfortunately this excellent book had not been  translated into Russian. Nevertheless  I was lucky to get a microfilm of the book.  Every day during a few months I was visiting a local library which had a special device to read microfilms.

 By the time I finished  reading the book I already knew what I was  going to do: I  decided to  develop a similar "constructive theory of Interpolation Spaces"   through  replacing a single one-parameter semi-group of operators by a representation of a general Lie group in a Banach space.

I have to say that the  book by P. Butzer and H. Berens is an excellent introduction to a number of topics in classical harmonic analysis. In particular  it contains the first systematic treatment of the theory of Intermediate Spaces and a very detailed and application oriented treatment of the theory of Semi-Groups of Operators. Both these subjects  were considered as "hot" topics at the end of 60s (see for example \cite{KPet}, \cite{K}).  In many ways  this book is still up to date and I always recommend it to my younger colleagues. 

 Some time later I was influenced by the classical  work  of Paul with K. Scherer \cite{BS}, \cite{BSa}  in which they greatly clarified the relationships between the  Interpolation and Approximation spaces and by Paul's pioneering paper  with H. Berens and S. Pawelke \cite{BBP} about approximation on spheres.

 Mathematics I learned from Paul Butzer helped me to become a graduate student of Selim Krein another world level expert in Interpolation spaces and applications of semigroups to differential equations \cite{KPet}-\cite{KPS}.
 
 Many years later when I came to US and became interested in  sampling theory  I found  out that Paul had already been  working in this field for a number of years and I learned a lot from insightful and stimulating  work written by P. Butzer, W. Splettst\"o\"ser and R. Stens  \cite{BSS88}.

 My interactions with Paul Butzer's work shaped my entire mathematical life and the list of some of my papers  \cite{Pes75}-\cite{Pes11} is the best evidence of it.

\textit{This is what I mean when I say that Paul Butzer is my teacher.}

In conclusion I would like to mention   that  it was my discussions with Paul Butzer and Gerhard Schmeisser of their beautiful  work with Rudolf Stens   \cite{BSS},  that stimulated my interest in the topic of the present  paper. I am very grateful to them for this.

\bigskip

\section{Introduction}

 Consider a trigonometric
polynomial $ P(t)$ of one variable $t$ of order $n$ as a function on a
unit circle $\mathbb{T}$. For its derivative
$P^{'}(t)$ the so-called Bernstein 
inequality holds true
\begin{equation}
\left\|P^{'}\right\|_{L_{p}(\mathbb{T})}\leq n
\| P\|_{L_{p}(\mathbb{T})},\>\>\>1\leq p\leq \infty.
\end{equation} M. Riesz \cite{Riesz1},  \cite{Riesz} states that the Bernstein inequality is equivalent to what is known today as the Riesz interpolation formula

\begin{equation}
P^{'}(t)=\frac{1}{4\pi}\sum_{k=1}^{2n}(-1)^{k+1}
\frac{1}{\sin^{2}\frac{t_{k}}{2}} \>\>P(t+t_{k}), \>\>\>t\in
\mathbb{S}, \>\>\>t_{k}=\frac{2k-1}{2n}\pi.
\end{equation}
The next formula holds true for functions in the Bernstein space $\mathbf{B}_{\sigma}^{p}, \>\>1\leq p\leq \infty$ which is comprised of  all entire functions of exponential type $\sigma$ which belong to $L_{p}(\mathbb{R})$ on the real line.
\begin{equation}\label{Boas-0}
f^{'}(t)=\frac{\sigma}{\pi^{2}}\sum_{k\in\mathbb{Z}}\frac{(-1)^{k-1}}{(k-1/2)^{2}}\>\>
\textit{f}\left(t+\frac{\pi}{\sigma}(k-1/2)\right),\>\>\> t\in \mathbb{R}.
\end{equation}
This formula was obtained by R.P. Boas \cite{Boas}, \cite{Boas1} and is known as Boas or generalized Riesz formula.  Again, like in periodic case this formula is equivalent to the Bernstein inequality in $L_{p}(\mathbb{R})$
$$
\left\|f^{'}\right\|_{L_{p}(\mathbb{R})}\leq \sigma\left\|f\right\|_{L_{p}(\mathbb{R})}.
$$

Recently, in the interesting papers \cite{Schm} and \cite{BSS} among other important results the Boas-type formula (\ref{Boas-0}) was generalized to higher order. 

 In particular it was shown that for $f\in B_{\sigma}^{\infty},\>\>\>\sigma>0,$ the following formulas hold 
$$
f^{(2m-1)}(t)=\left(\frac{\sigma}{\pi}\right)^{2m-1}\sum_{k\in \mathbb{Z}}(-1)^{k+1}A_{m,k}f\left(t+\frac{\pi}{\sigma}(k-\frac{1}{2})\right), \>\>\>m\in \mathbb{N},
$$
$$
f^{(2m)}(t)=\left(\frac{\sigma}{\pi}\right)^{2m}\sum_{k\in \mathbb{Z}}(-1)^{k+1}B_{m,k}f\left(t+\frac{\pi}{\sigma}k\right), \>\>\>m\in \mathbb{N},
$$
where 
\begin{equation}\label{A}
A_{m,k}=(-1)^{k+1} \rm sinc ^{(2m-1)}\left(\frac{1}{2}-k\right)=
$$
$$
\frac{(2m-1)!}{\pi(k-\frac{1}{2})^{2m}}\sum_{j=0}^{m-1}\frac{(-1)^{j}}{(2j)!}\left(\pi(k-\frac{1}{2})\right)^{2j},\>\>\>m\in \mathbb{N},
\end{equation}
for $k\in \mathbb{Z}$ and 
\begin{equation}\label{B}
B_{m,k}=(-1)^{k+1} \rm sinc ^{(2m)}(-k)=\frac{(2m)!}{\pi k^{2m+1}}\sum_{j=0}^{m-1}\frac{(-1)^{j}(\pi k)^{2j+1}}{(2j+1)!},\>\>\>m\in \mathbb{N},\>\>\>k\in \mathbb{Z}\setminus {0},
\end{equation}
and 
\begin{equation}\label{B0}
B_{m,0}=(-1)^{m+1} \frac{\pi^{2m}}{2m+1},\>\>\>m\in \mathbb{N}.
\end{equation}

Let us remind that \rm sinc(t) is defined as $\frac{\sin \pi t}{\pi t}$, if $t\neq 0$, and $1$, if $t=0$.

To illustrate our results  let us assume that we are given an operator $D$ that generates a strongly continuous  group of isometries $e^{tD}$ in a Banach space $E$.  
\begin{definition}
The  subspace of exponential vectors $\mathbf{E}_{\sigma}(D), \>\>\sigma\geq 0,$ is defined as  a set of all vectors $f$ in $E$ 
 which belong to $\mathcal{D}^{\infty}=\bigcap_{k\in \mathbb{N}}\mathcal{D}^{k}$, where $\mathcal{D}^{k}$ is the domain of $D^{k}$,  and for which there exists a constant $C(f)>0$ such that 
\begin{equation}\label{Bernstein}
\|D^{k}f\|\leq C(f)\sigma^{k}, \>\>k\in \mathbb{N}.
\end{equation}
 
\end{definition}
Note, that every $\mathbf{E}_{\sigma}(D)$ is clearly a linear subspace of $E$. What is really  important is the fact that union of all $\mathbf{E}_{\sigma}(D)$ is dense in $E$ (Theorem \ref{Density}). 

\begin{remark}
It is worth to stress  that if $D$ generates a strongly continuous bounded semigroup then the set  $\bigcup_{\sigma\geq 0}\mathbf{E}_{\sigma}(D)$ may not be  dense in $E$. 

Indeed, (see  \cite{Nelson}) consider a strongly continuous bounded semigroup $T(t)$ in $L_{2}(0,\infty)$ defined for every $f\in L_{2}(0,\infty)$ as $T(t)f(x)=f(x-t),$ if $\>x\geq t$ and $T(t)f(x)=0,$ if $\>\>0\leq x<t$.   If $ f\in \mathbf{E}_{\sigma}(D)$ then for any $g\in L_{2}(0,\infty)$ the function $\left<T(t)f,\>g\right>$ is analytic in $t$ (see below section \ref{sec2}). Thus if $g$ has compact support then $\left<T(t)f,\>g\right>$ is zero for  all  $t$ which implies that $f$ is zero. In other words in this case every space $\mathbf{E}_{\sigma}(D)$ is trivial. 

\end{remark}

One of our results is that a vector $ f$ belongs to $ \mathbf{E}_{\sigma}(D)$ if and only if  the following sampling-type formulas hold
\begin{equation}\label{sam}
e^{tD}D^{2m-1}f=\left(\frac{\sigma}{\pi}\right)^{2m-1}\sum_{k\in \mathbb{Z}}(-1)^{k+1}A_{m,k}e^{\left(t+\frac{\pi}{\sigma}(k-1/2)   \right)D}f,\>\>\>m\in \mathbb{N},
\end{equation}

\begin{equation}\label{sam2222}
e^{tD}D^{2m}f=\left(\frac{\sigma}{\pi}\right)^{2m}\sum_{k\in \mathbb{Z}}(-1)^{k+1}B_{m,k}e^{\left(t+\frac{\pi k}{\sigma}    \right)D}f, \>\>\>m\in \mathbb{N},
\end{equation}
Which are equivalent to  the following Boas-type formulas
\begin{equation}\label{boas-0}
D^{2m-1}f=\left(\frac{\sigma}{\pi}\right)^{2m-1}\sum_{k\in \mathbb{Z}}(-1)^{k+1}A_{m,k}e^{\left(   \frac{\pi}{\sigma}(k-1/2)   \right)D}f,\>\>\>m\in \mathbb{N},\>\>\> f\in \mathbf{E}_{\sigma}(D),
\end{equation}
and
\begin{equation}\label{boas-01}
D^{2m}f=\left(\frac{\sigma}{\pi}\right)^{2m}\sum_{k\in \mathbb{Z}}(-1)^{k+1}B_{m,k}e^{\frac{\pi k}{\sigma}D}f,\>\>\>m\in \mathbb{N}\cup{0},\>\>\> f\in \mathbf{E}_{\sigma}(D).
\end{equation}
The formulas (\ref{sam}) and (\ref{sam2222}) are  a sampling-type formulas in the sense that they provide   explicit expressions  for a  trajectory $e^{tD}D^{k}f$ with $f\in \mathbf{E}_{\sigma}(D)$ in terms of a countable number of equally spaced samples of  trajectory of $f$.

Note, that \rm since $e^{tD} ,\>\>\>t\in \mathbb{R}, $ is a group of operators any trajectory $e^{tD}f,\>\>\>f\in E,$ is completely determined by any (single) sample $e^{t_{0}D}f,$ because for any $t\in \mathbb{R}$
$$
e^{tD}f=e^{(t-t_{0})D} \left(e^{t_{0}D}f\right).
$$
 The formulas (\ref{sam}) and (\ref{sam2222}) have,  however, a different  nature: they represent a trajectory as a "linear combination" of a countable number of samples.

It seems to be very interesting that an operator and the group can be rather sophisticated (think, for example, about a Schr\"{o}dinger operator $D=-\Delta+V(x)$ and the corresponding group $e^{itD}$ in $L_{2}(\mathbb{R}^{d})$). However, formulas (\ref{sam})-(\ref{boas-01}) are universal in the sense that they contain the same coefficients and the same sets of sampling points.

We are making a list of some  important  properties of  the Boas-type interpolation formulas  (compare to  \cite{BSS}):

\begin{enumerate}

\item  The formulas  hold for  vectors $f$ in the set $\bigcup_{\sigma\geq 0}\mathbf{E}_{\sigma}(D)$ which is dense in $E$ (see Theorem \ref{Density}).

\item The sample points $ \frac{\pi}{\sigma}(k-1/2) $ are uniformly spaced according to the Nyquist rate and are independent on $f$.

\item  The coefficients do not depend on $f$.

\item The coefficients decay like $O(k^{-2})$ as $k$ goes to infinity.

\item  In formulas (\ref{boas-0}) and (\ref{boas-01}) one has unbounded operators (in general case) on the left-hand side and bounded operators on the right-hand side.

\item There is a number of interesting relations between Boas-type formulas, see below (\ref{sigmaB}), (\ref{powerB}).
 
\end{enumerate}

Our main objective is to obtain  a set of new  formulas for one-parameter groups which appear when one considers representations of Lie groups (see section \ref{App}). 
Note, that  generalizations of (\ref{Boas-0}) with applications to compact homogeneous manifolds were initiated in \cite{Pes88}.

In  our applications we deal  with a set of non-commuting generators $D_{1},...,D_{d}$. In subsection \ref{compact}  these operators come from a representation of a compact Lie group and we are able to show that 
$
\cup_{\sigma\geq 0}\cup_{1\leq j\leq d} \mathbf{E}_{\sigma}(D_{j})$ is dense  in all appropriate Lebesgue spaces.
  We cannot prove  a similar  fact 
in the next subsection  \ref{non-compact} in which we consider a non-compact Heisenberg group. Moreover, in subsection \ref{Schrod} in which the Schr\"{o}dinger representation is discussed we note that this property  does  not  hold in general.

\section{ Boas-type   formulas for exponential  vectors}\label{sec2}

We assume that $D$ is a generator of one-parameter group of
isometries $e^{ tD}$ in a Banach space $E$ with the norm $\|\cdot
\|$. 

\begin{definition}
The Bernstein subspace
 $\mathbf{B}_{\sigma}(D), \>\>\sigma\geq 0,$ is defined as  a set of all vectors $f$ in $E$ 
 which belong to $\mathcal{D}^{\infty}=\bigcap_{k\in \mathbb{N}}\mathcal{D}^{k}$, where $\mathcal{D}^{k}$ is the domain of $D^{k}$ and for which 
\begin{equation}\label{Bernstein}
\|D^{k}f\|\leq \sigma^{k}\|f\|, \>\>k\in \mathbb{N}.
\end{equation}
\end{definition}

It is obvious that  $\mathbf{B}_{\sigma}(D)\subset \mathbf{E}_{\sigma}(D) , \>\>\>\sigma\geq 0, $
However, it is not even clear that  $\mathbf{B}_{\sigma}(D), \>\>\sigma\geq 0,$ is a linear subspace. It follows from the following interesting fact.

\begin{theorem}\label{t1}
  Let $D$ be a generator of an one parameter group of operators
$e^{tD }$ in a Banach space $E$ and $\|e^{tD}f\|=\|f\|$.  Then for every $\sigma\geq 0$ 
$$
\mathbf{B}_{\sigma}(D)= \mathbf{E}_{\sigma}(D) , \>\>\>\sigma\geq 0, 
$$

\end{theorem}

\begin{proof}
If  $f\in \mathbf{E}_{\sigma}(D)$  then  for
any complex number $z$ we have

$$ 
\left\|e^{zD}f\right\|=\left\|\sum ^{\infty}_{r=0}(z^{r}D^{r}f)/r!\right\|\leq C(f) \sum
^{\infty}_{r=0}|z|^{r}\sigma^{r}/r!=C(f)e^{|z|\sigma}.
$$
It implies that for any functional $\psi^{*}\in E^{*}$ the scalar function
$\left<e^{zD}f,\psi^{*}\right>$ is an entire function
 of exponential type $\sigma $ which is bounded on the real axis 
by the constant $\|\psi^{*}\| \|f\|$.
 An application of the classical Bernstein inequality gives

$$\left\|\left<e^{tD}D^{k}f,\psi^{*}\right>\right\|_{C(R^{1})}=\left\|\left(\frac{d}{dt}\right)^{k}\left<e^{tD}f,\psi^{*}\right>\right\|_{
C(R^{1})} \leq\sigma^{k}\|\psi^{*}\| \|f\|.$$
From here  for $t=0$ we obtain 
$$ \left|\left<D^{k}f,\psi^{*}\right>\right|\leq \sigma ^{k} \|\psi^{*}\| \|f\|.
$$
Choice of $\psi^{*}\in E^{*}$ such that $\|\psi^{*}\|=1$ and $\left<D^{k}f,\psi^{*}\right>=\|D^{k}f\|$ gives
the inequality $\|D^{k}f\|\leq
 \sigma ^{k} \|f\|,\>\> k\in \mathbb{N}$, which implies Theorem.
\end{proof}

\begin{remark}
We just  mention that 
in the important case of a self-adjoint operator $D$ in a Hilbert space $E$ there is a way to describe Bernstein vectors in terms of  a spectral Fourier transform or in terms of the spectral measure associated with $D$ (see \cite{KPes},  \cite{Pes88}-\cite{Pes11}  for more details). 
\end{remark}

Let's introduce bounded operators

\begin{equation}\label{b1}
\mathcal{B}_{D}^{(2m-1)}(\sigma)f=\left(\frac{\sigma}{\pi}\right)^{2m-1}\sum_{k\in \mathbb{Z}}(-1)^{k+1}A_{m,k}e^{\frac{\pi}{\sigma}(k-1/2)D}f,\>\> f\in E, \>\>\sigma>0,\>\>\>m\in \mathbb{N},
\end{equation}
\begin{equation}\label{b2}
\mathcal{B}_{D}^{(2m)}(\sigma)f=\left(\frac{\sigma}{\pi}\right)^{2m}\sum_{k\in \mathbb{Z}}(-1)^{k+1}B_{m,k}e^{\frac{\pi k}{\sigma}D}f, \>\>f\in E,\> \sigma>0,\>m\in \mathbb{N},
\end{equation}
where $A_{m,k}$ and $B_{m,k}$ are defined in (\ref{A})-(\ref{B0}).  Both series converge in $E$ due to the following formulas (see \cite{BSS})

\begin{equation}
\left(\frac{\sigma}{\pi}\right)^{2m-1}\sum_{k\in \mathbb{Z}}\left|A_{m,k}\right|=\sigma^{2m-1},\>\>\>\>\>
 \left(\frac{\sigma}{\pi}\right)^{2m}\sum_{k\in \mathbb{Z}}\left|B_{m,k}\right|=\sigma^{2m}\label{id-2}.
 \end{equation}
 Since $\|e^{tD}f\|=\|f\|$ it implies that 
 
 \begin{equation}\label{norms}
 \|\mathcal{B}_{D}^{(2m-1)}(\sigma)f\|\leq \sigma^{2m-1}\|f\|,\>\>\>\>\>\|\mathcal{B}_{D}^{(2m)}(\sigma)f\|\leq \sigma^{2m}\|f\|,\>\>\>f\in E. 
 \end{equation}
 \begin{theorem}
 If $D$ generates a one-parameter strongly continuous bounded group of operators $e^{tD}$ in a Banach space $E$ then  the following conditions are equivalent:

\begin{enumerate}

\item  $f$ belongs to $\mathbf{B}_{\sigma}(D)$.

\item   The abstract-valued function $e^{tD}f$  is an entire function of exponential type $\sigma$ which is bounded on the real line.

\item The following Boas-type interpolation formulas hold true for $r\in \mathbb{N}$

\begin{equation}\label{B1}
D^{r}f=\mathcal{B}_{D}^{(r)}(\sigma)f,\>\>\>\>\>f\in \mathbf{B}_{\sigma}(D).
\end{equation}

\end{enumerate}

\end{theorem}

\begin{proof}
The proof of Theorem \ref{t1} shows that 1) $\rightarrow$ 2).   Then obviously for any  $\psi^{*}\in E^{*}$
 the  function $
F(t)=\left<e^{tD}f, \psi^{*}\right>
$
is of exponential type $\sigma$ and bounded on $\mathbb{R}.$ Thus by  \cite{BSS} we have
$$
F^{(2m-1)}(t)=\left(\frac{\sigma}{\pi}\right)^{2m-1}\sum_{k\in \mathbb{Z}}(-1)^{k+1}A_{m,k}F\left(t+\frac{\pi}{\sigma}(k-1/2)    \right),\>\>\>m\in \mathbb{N},
$$
$$
F^{(2m)}(t)=\left(\frac{\sigma}{\pi}\right)^{2m}\sum_{k\in \mathbb{Z}}(-1)^{k+1}B_{m,k}F\left(t+\frac{\pi k}{\sigma}    \right),\>\>\>m\in \mathbb{N}.
$$
Together with 
$$
\left(\frac{d}{dt}\right)^{k}F(t)=\left<D^{k}e^{tD}f,\psi^{*}\right>,
$$
  it shows
$$
\left<e^{tD}D^{2m-1}f, \psi^{*}\right>=\left(\frac{\sigma}{\pi}\right)^{2m-1}\sum_{k\in \mathbb{Z}}(-1)^{k+1}A_{m,k}\left<e^{\left(t+\frac{\pi}{\sigma}(k-1/2)    \right)D}f,\>\>\psi^{*}\right>,\>\>\>m\in \mathbb{N},
$$
and also
$$
\left<e^{tD}D^{2m}f, \psi^{*}\right>=\left(\frac{\sigma}{\pi}\right)^{2m}\sum_{k\in \mathbb{Z}}(-1)^{k+1}B_{m,k}\left<e^{\left(t+\frac{\pi k}{\sigma}    \right)D}f,\>\> \psi^{*}\right>,\>\>\>m\in \mathbb{N}.
$$
Since both series (\ref{b1}) and (\ref{b2}) converge in $E$ and the last two equalities hold for any $\psi^{*}\in E$ we obtain the next two formulas

\begin{equation}\label{sam1}
e^{tD}D^{2m-1}f=\left(\frac{\sigma}{\pi}\right)^{2m-1}\sum_{k\in \mathbb{Z}}(-1)^{k+1}A_{m,k}e^{\left(t+\frac{\pi}{\sigma}(k-1/2)   \right)D}f,\>\>\>m\in \mathbb{N},
\end{equation}
\begin{equation}\label{sam2}
e^{tD}D^{2m}f=\left(\frac{\sigma}{\pi}\right)^{2m}\sum_{k\in \mathbb{Z}}(-1)^{k+1}B_{m,k}e^{\left(t+\frac{\pi k}{\sigma}    \right)D}f, \>\>\>m\in \mathbb{N}.
\end{equation}
In turn when $t=0$ these formulas become formulas (\ref{B1}). 

The fact that 3) $\rightarrow$ 1) easily follows from  the formulas (\ref{B1})
and (\ref{norms}).  
Theorem is proved.
\end{proof}
\begin{corollary}
Every $\mathbf{B}_{\sigma}(D)$ is a closed linear subspace of $E$.
\end{corollary}

\begin{corollary}
If $f$ belongs to  $\mathbf{B}_{\sigma}(D)$ then for any $\sigma_{1}\geq \sigma,\>\>\>\sigma_{2}\geq \sigma$ one has 
\begin{equation}\label{sigmaB}
\mathcal{B}^{(r)}_{D}(\sigma_{1})f=\mathcal{B}^{(r)}_{D}(\sigma_{2})f,\>\>\>r\in \mathbb{N}.
\end{equation}
\end{corollary}

 Let us introduce the  notation
 $$
  \mathcal{B}_{D}(\sigma)= \mathcal{B}_{D}^{(1)}(\sigma).
 $$
 One has the following "power" formula which easily follows from the fact that operators $
  \mathcal{B}_{D}(\sigma)$ and $D$ commute on any $\mathbf{B}_{\sigma}(D)$.
 \begin{corollary}
 For any $r\in \mathbb{N}$ and  any $f\in \mathbf{B}_{\sigma}(D)$
 \begin{equation}\label{powerB}
 D^{r}f=\mathcal{B}_{D}^{(r)}(\sigma)f=\mathcal{B}_{D}^{r}(\sigma)f,
 \end{equation}
 where $\mathcal{B}_{D}^{r}(\sigma)f=\mathcal{B}_{D}(\sigma)...\mathcal{B}_{D}(\sigma)f.$
\end{corollary}

Let us introduce the following notations
$$\mathcal{B}^{(2m-1)}_{D}(\sigma,N)f=
\left(\frac{\sigma}{\pi}\right)^{2m-1}\sum_{|k|\leq N}(-1)^{k+1}A_{m,k}e^{\frac{\pi}{\sigma}(k-1/2)D_{j}}f,
$$
$$
\mathcal{B}^{(2m)}_{D}(\sigma, N)f=
\left(\frac{\sigma}{\pi}\right)^{2m}\sum_{|k|\leq N}(-1)^{k+1}B_{m,k}e^{\frac{\pi k}{\sigma}D_{j}}f.
$$
One obviously has the following set of approximate Boas-type formulas.

\begin{theorem}
If $f\in \mathbf{B}_{\sigma}(D)$ and $r\in \mathbb{N}$ then

\begin{equation}\label{appB}
D^{(r)}f=\mathcal{B}_{D}^{(r)}(\sigma, N)f+O(N^{-2}).
\end{equation}

\end{theorem}
The next Theorem contains another Boas-type formula.

\begin{theorem}
If $f\in \mathbf{B}_{\sigma}(D)$ then the following sampling formula holds for $t\in \mathbb{R}$ 
and  $n\in \mathbb{N}$
\begin{equation}\label{s2}
e^{tD}D^{n}f=\sum_{k}\frac{e^{\frac{k\pi}{\sigma}D}f-f}{\frac{k\pi}{\sigma}}\left\{  n \>\rm sinc^{(n-1)}\left(\frac{\sigma t}{\pi}-k\right)  +\frac{\sigma t}{\pi} \rm sinc^{(n)}\left(\frac{\sigma t}{\pi}-k\right)
   \right\}.
\end{equation}
In particular, for $n\in \mathbb{N}$
one has 
\begin{equation}\label{Q}
D^{n}f=\mathcal{Q}_{D}^{n}(\sigma)f,
\end{equation}
where the bounded operator $\mathcal{Q}_{D}^{n}(\sigma)$ is given by the formula
\begin{equation}
\mathcal{Q}_{D}^{n}(\sigma)f=n\sum_{k}\frac{e^{\frac{k\pi}{\sigma}D}f-f}{\frac{k\pi}{\sigma}}\left[   \rm sinc^{(n-1)}\left(-k\right)  +\rm sinc^{(n)}\left(-k\right)
   \right].
\end{equation}
\end{theorem}

\begin{proof}

         If $f\in \mathbf{B}_{\sigma}(D)$ then for any $g^{*}\in E^{*}$ the function $F(t)=\left<e^{tD}f,\>g^{*}\right>$ belongs to $B_{\sigma}^{\infty}(\mathbb{R})$.

We consider $F_{1}\in B_{\sigma}^{2}( \mathbb{R}),$ which is defined as follows.
If $t\neq 0$ then 
\begin{equation}\label{F}
F_{1}(t)=\frac{F(t)-F(0)}{t}=\left<\frac{e^{tD}f-f}{t},\>g^{*}\right>,
\end{equation}
and if $t=0$ then
$
F_{1}(t)=\frac{d}{dt}F(t)|_{t=0}=\left<Df,\>g^{*}\right>.
$
We have
$$
F_{1}(t)=\sum_{k}F_{1}\left(\frac{k\pi}{\sigma}\right)\> \rm sinc\left(\frac{\sigma t}{\pi}-k\right).
$$
From here we obtain the next formula
$$
\left(\frac{d}{dt}\right)^{n}F_{1}(t)=\sum_{k}F_{1}\left(\frac{k\pi}{\sigma} \right) \rm sinc^{(n)}\left(\frac{\sigma t}{\pi}-k\right)
$$
and \rm since 
$$
\left(\frac{d}{dt}\right)^{n}F(t)=n\left(\frac{d}{dt}\right)^{n-1}F_{1}(t)+t\left(\frac{d}{dt}\right)^{n}F_{1}(t)
$$
we obtain 
$$
\left(\frac{d}{dt}\right)^{n}F(t)=n\sum_{k}F_{1}\left(\frac{k\pi}{\sigma}\right)\> \rm sinc^{(n-1)}\left(\frac{\sigma t}{\pi}-k\right)+\frac{\sigma t}{\pi}\sum_{k}F_{1}\left(\frac{k\pi}{\sigma}\right)\> \rm sinc^{(n)}\left(\frac{\sigma t}{\pi}-k\right)
$$
Since
$
\left(\frac{d}{dt}\right)^{n}F(t)=\left<D^{n}e^{tD}f, g^{*}\right>,
$
and
$$
F_{1}\left(\frac{k\pi}{\sigma}\right)=\left<\frac{e^{\frac{k\pi}{\sigma}D}f-f}{\frac{k\pi}{\sigma}}, g^{*}\right>
$$
we obtain  that  for $t\in \mathbb{R},\>\>n\in \mathbb{N},$
$$
D^{n}e^{tD}f=\sum_{k}\frac{e^{\frac{k\pi}{\sigma}D}f-f}{\frac{k\pi}{\sigma}}\left[  n\> \rm sinc^{(n-1)}\left(\frac{\sigma t}{\pi}-k\right)  +\frac{\sigma t}{\pi} \rm sinc^{(n)}\left(\frac{\sigma t}{\pi}-k\right)
   \right].
$$ 
Theorem is proved. 
\end{proof}

The next Theorem shows that Boas-type formulas make sense for a dense set of vectors. 
\begin{theorem}\label{Density}
The set  $\bigcup_{\sigma\geq 0}\mathbf{B}_{\sigma}(D)$ is dense in $E$.
\end{theorem}

\begin{proof}

Note that if  $\phi\in L_{1}(\mathbb{R}),\>\>\>\|\phi\|_{1}=1,$ is an entire function of exponential
type $\sigma$ then for any $f\in E$ the vector
$$
g=\int _{-\infty}^{\infty}\phi(t)e^{tD}fdt
$$
belongs to $\mathbf{B}_{\sigma}(D).$ Indeed,  for every real $\tau$ we
have
$$
e^{\tau
D}g=\int_{-\infty}^{\infty}\phi(t)e^{(t+\tau)D}fdt=\int_{-\infty}^{\infty}\phi(
t-\tau)e^{tD}fdt.
$$
 Using this formula we can extend the abstract function $e^{\tau D}g$ to the
complex plane as
$$
e^{zD}g=\int_{-\infty}^{\infty}\phi(t-z)e^{tD}fdt.
$$
 Since by assumption  $h$ is an entire function of exponential
type $\sigma$ and $\|\phi\|_{L_{1}(\mathbb{R})}=1$
 we have
$$
\|e^{zD}g\|\leq
\|f\|\int_{-\infty}^{\infty}|\phi(t-z)|dt\leq\|f\|e^{\sigma|z|}.
 $$
 This inequality  implies that $g$ belongs to $\mathbf{B}_{\sigma}(D)$.

Let
 $$
h(t)=a\left(\frac{\sin (t/4)}{t}\right)^{4}
$$
 and
$$
a=\left(\int_{-\infty}^{\infty}\left(\frac{\sin
(t/4)}{t}\right)^{4}dt\right)^{-1}.
$$
Function $h$ will have the
 following properties:

\begin{enumerate}

\item $h$ is an even nonnegative entire function of exponential
type one;

\item $h$ belongs to $L_{1}(\mathbb{R})$ and its
$L_{1}(\mathbb{R})$-norm is $1$;

\item  the integral
\begin{equation}\int_{-\infty}^{\infty}h(t)|t|dt
\end{equation} 
is finite.
\end{enumerate}
Consider
 the following vector 
\begin{equation}
\mathcal{R}_{h}^{\sigma}(f)=
\int_{-\infty}^{\infty}h(t)e^{\frac{t}{\sigma}D
} fdt=\int_{-\infty}^{\infty}h(t\sigma)e^{tD}fdt,\label{id}
 \end{equation}
Since the function $h(t)$ has exponential type one the function
$h(t\sigma)$ has the type $\sigma$. It  implies (by the previous) that $\mathcal{R}_{h}^{\sigma}(f)$ belongs to $\mathbf{B}_{\sigma}(D)$.

The modulus of
continuity is defined as in \cite{BB}
$$
\Omega(f,s)=\sup_{|\tau|\leq
s}\left\|\Delta_{\tau}f\right\|,\>\>\>\>
 \Delta_{\tau}f=(I-e^{\tau D})f. \label{dif}
$$
Note, that for every $f\in E$ the modulus $\Omega(f,s)$ goes to zero when $s$ goes to zero.  Below we are using  an easy verifiable inequality
$
\Omega\left(f, as\right)\leq \left(1+a\right)\Omega(f,
s),\>\>\> a\in \mathbb{R}_{+}.
$
We obtain
$$
\|f-\mathcal{R}_{h}^{\sigma}(f)\|\leq
\int_{-\infty}^{\infty}h(t)\left\|\Delta_{t/\sigma}f\right\|dt\leq
\int_{-\infty}^{\infty}h(t)\Omega\left(f, t/\sigma\right)dt\leq 
$$
$$
\Omega\left(f,
\sigma^{-1}\right)\int_{-\infty}^{\infty}h(t)(1+|t|)dt\leq
{C}_{h}\Omega\left(f,
\sigma^{-1}\right),
$$
where the integral
 $$
C_{h}=\int_{-\infty}^{\infty}h(t)(1+|t|)dt
$$
 is finite by the choice of $h$. Theorem  is proved.
\end{proof}

\section{Analysis on compact homogeneous manifolds}

 Let $\mathbf{M}, \>\>\rm dim \>\mathbf{M}=m,$ be a
compact connected $C^{\infty}$-manifold. One says  that a compact
Lie group $G$ effectively acts on $\mathbf{M}$ as a group of
diffeomorphisms if the following holds true: 

\begin{enumerate}

\item   Every element $g\in G$ can be identified with a diffeomorphism
$
g: \mathbf{M}\rightarrow \mathbf{M}
$
of $\mathbf{M}$ onto itself and
$
g_{1}g_{2}\cdot x=g_{1}\cdot(g_{2}\cdot x), g_{1}, g_{2}\in G,
x\in \mathbf{M},
$
where $g_{1}g_{2}$ is the product in $G$ and $g\cdot x$ is the
image of $x$ under $g$.

\item The identity $e\in G$ corresponds to the trivial diffeomorphism
 $
e\cdot x=x.
$
\item For every $g\in G, g\neq e,$ there exists a point $x\in \mathbf{M}$ such
that $g\cdot x\neq x$.

\end{enumerate}

A group $G$ acts on $\mathbf{M}$ \textit{transitively} if in addition to
1)- 3) the following property holds:
4) for any two points $x,y\in \mathbf{M}$ there exists a diffeomorphism
$g\in G$ such that
$\>\>\>
g\cdot x=y.
$

A \textit{homogeneous} compact manifold $\mathbf{M}$ is an
$C^{\infty}$-compact manifold on which transitively acts a compact
Lie group $G$. In this case $\mathbf{M}$ is necessary of the form $G/K$,
where $K$ is a closed subgroup of $G$. The notation $L_{p}(\mathbf{M}),
1\leq p\leq \infty,$ is used for the usual Banach  spaces
$L_{p}(\mathbf{M},dx), 1\leq p\leq \infty$, where $dx$ is the normalized invariant
measure.

Every element $X$ of the Lie algebra of $G$ generates a vector
field on $\mathbf{M}$ which we will denote by the same letter $X$. Namely,
for a smooth function $f$ on $\mathbf{M}$ one has
$$
 Xf(x)=\lim_{\tau\rightarrow 0}\frac{f(\exp \tau X \cdot x)-f(x)}{\tau}
 $$
for every $x\in \mathbf{M}$. In the future we will consider on $\mathbf{M}$ only
such vector fields. Translations along integral curves of such
vector field $X$ on $\mathbf{M}$  can be identified with a one-parameter
group of diffeomorphisms of $\mathbf{M}$ which is usually denoted as $\exp
\tau X, -\infty<\tau<\infty$. At the same time the one-parameter group
$\exp \tau X, -\infty<\tau<\infty,$ can be treated as a strongly
continuous one-parameter group of operators in a space $L_{p}(\mathbf{M}),
1\leq p\leq \infty$ which acts on functions according to the
formula
$
f\rightarrow f(\exp \tau X\cdot x), \tau \in \mathbb{R}, f\in L_{p}(\mathbf{M}),
x\in \mathbf{M}.
$
 The
generator of this one-parameter group will be denoted as $D_{X,p}$
and the group itself will be denoted as
$
e^{\tau D_{X,p}}f(x)=f(\exp \tau X\cdot x), t\in \mathbb{R}, f\in
L_{p}(\mathbf{M}), x\in \mathbf{M}.
$

According to the general theory of one-parameter groups in Banach
spaces \cite{BB}, Ch. I,  the operator $D_{X,p}$ is a closed
operator in every $L_{p}(\mathbf{M}), 1\leq p\leq \infty.$ In order to
simplify notations we will often use notation $D_{X}$ instead of
$D_{X, p}$.

It is known \cite{H2},
 Ch. V,  that
 on every compact homogeneous manifold $\mathbf{M}=G/K$ there exist vector fields
  $X_{1},X_{2},..., X_{d},$ $ \>\>\>d=\rm dim\>\> G,$ such that the second order  differential
   operator 
   $$
X_{1}^{2}+ X_{2}^{2}+ ...+ X_{d}^{2}, \>\>\>d=\rm dim\> G,
   $$
   commutes with all vector fields $X_{1},...,X_{d}$ on $\mathbf{M}$. The
   corresponding operator in $L_{p}(\mathbf{M}), 1\leq p\leq\infty,$
\begin{equation}
-\mathcal{L}=D_{1}^{2}+ D_{2}^{2}+ ...+ D_{d}^{2},\>\>\>
D_{j}=D_{X_{j}}, \>\>\>d=\rm dim\> G,\label{Laplacian}
\end{equation}
commutes with all operators $D_{j}=D_{X_{j}}$. This operator
$\mathcal{L}$ which is usually called the Laplace operator is
involved in most of constructions and results of our paper.

 The  operator $\mathcal{L}$  is an elliptic differential operator which
is defined on $C^{\infty}(\mathbf{M})$ and we will use the same notation
$\mathcal{L}$ for its closure from $C^{\infty}(\mathbf{M})$ in $L_{p}(\mathbf{M}),
1\leq p\leq \infty$. In the case $p=2$ this closure is a
self-adjoint positive definite operator in the space $L_{2}(\mathbf{M})$.
The spectrum of this operator is discrete and goes to infinity
$0=\lambda_{0}<\lambda_{1}\leq \lambda_{2}\leq ...$, where
 we count each eigenvalue with its  multiplicity.  For eigenvectors corresponding 
 to eigenvalue $\lambda_{j}$ we will use notation $\varphi_{j}$, i. e. 
$$
 \mathcal{L}\varphi_{j}=\lambda_{j}\varphi_{j}.
$$
The spectrum and the set of eigenfunctions of $\mathcal{L}$ are the same in all spaces $L_{p}(\mathbb{S}^{d})$.

Let $\varphi_{0}, \varphi_{1}, \varphi_{2}, ...$ be a corresponding
complete system of orthonormal eigenfunctions and
$\textbf{E}_{\sigma}(\mathcal{L}), \sigma>0,$ be a span of all
eigenfunctions of $\mathcal{L}$ whose corresponding eigenvalues
are not greater $\sigma$.

In the rest of the paper the notations $\mathbb{D}=\{D_{1},...,
D_{d}\}, d= \rm dim \>G,$ will be used for differential operators in
$L_{p}(\mathbf{M}), 1\leq p\leq \infty,$ which are involved in the formula
(\ref{Laplacian}).

\begin{definition}[\cite{Pes90}, \cite{Pes08}].
We say that a  function $f\in L_{p}(\mathbf{M}), 1\leq p\leq \infty,$
belongs to the Bernstein space
$\mathbf{B}_{\sigma}^{p}(\mathbb{D}),
\mathbb{D}=\{D_{1},...,D_{d}\}, d=\rm dim \>G,$ if and only if for every
$1\leq i_{1},...i_{k}\leq d$ the following Bernstein inequality
holds true
 \begin{equation}
 \|D_{i_{1}}...D_{i_{k}}f\|_{p}\leq
 \sigma^{k}\|f\|_{p}, k\in \mathbb{N}.\label{Bern}
\end{equation}
We say that a  function $f\in L_{p}(\mathbf{M}), 1\leq p\leq \infty,$
belongs to the Bernstein space
$\mathbf{B}_{\sigma}^{p}(\mathcal{L}), $ if and only if for every
$k\in \mathbb{N}$ the following Bernstein inequality holds true
$$
\|\mathcal{L}^{k}f\|_{p}\leq \sigma^{k}\|f\|_{p}, k\in \mathbb{N}.
$$
\end{definition}

Since $\mathcal{L}$ in the space $L_{2}(\mathbf{M})$ is self-adjoint and
positive-definite there exists a unique positive square root
$\mathcal{L}^{1/2}$. In this case the last inequality is
equivalent to the inequality
$
\|\mathcal{L}^{k/2}f\|_{2}\leq \sigma^{k/2}\|f\|_{2}, k\in
\mathbb{N}.
$ 

Note that at this point it is not clear if the Bernstein spaces
$\mathbf{B}_{\sigma}^{p}(\mathbb{D}),\>
\mathbf{B}_{\sigma}^{p}(\mathcal{L})$ are linear spaces. The
facts  that these spaces  are linear, closed and invariant (with respect to operators $D_{j}$ ) were  established in \cite{Pes08}.

It was shown in \cite{Pes08} that for 
$1\leq p, q\leq \infty$ the following equality holds true
$$
\mathbf{B}_{\sigma}^{p}(\mathbb{D})
=\mathbf{B}_{\sigma}^{q}(\mathbb{D})\equiv
\mathbf{B}_{\sigma}(\mathbb{D}),\>\>\>\mathbb{D}=\{D_{1},...,D_{d}\},
$$
which means that if the Bernstein-type inequalities (\ref{Bern})
are satisfied for a single  $1\leq p\leq \infty$, then they are
satisfied for all $1\leq p\leq \infty$.

\begin{definition}
$\mathcal{E}_{\lambda}(\mathcal{L}), \lambda>0,$ be a span of all
eigenfunctions of $\mathcal{L}$ whose corresponding eigenvalues
are not greater $\lambda$.
\end{definition}

 The following embeddings which describe relations between
Bernstein spaces $
\textbf{B}_{n}$ and eigen spaces $\mathcal{E}_{\lambda}(\mathcal{L})$  were proved in \cite{Pes08}
\begin{equation}\label{incl}
\textbf{B}_{\sigma}(\mathbb{D})\subset\mathcal{E}_{\sigma^{2}d}(\mathcal{L})\subset
\textbf{B}_{\sigma\sqrt{d}}(\mathbb{D}).
\end{equation} 
These embeddings obviously  imply the equality
$$
\bigcup_{\sigma>0} \textbf{B}_{\sigma}(\mathbb{D})=\bigcup_{\lambda}
\mathcal{E}_{\lambda}(\mathcal{L}),
$$
which means  \textit{that a function on $\mathbf{M}$ satisfies a Bernstein
inequality (\ref{Bern}) in the norm of $L_{p}(\mathbf{M}), 1\leq p\leq
\infty,$ if and only if it is a linear combination of
eigenfunctions of $\mathcal{L}$.}

As a consequence we obtain \cite{Pes08} the following Bernstein inequalities for $k\in \mathbb{N,}$
\begin{equation}
\|\mathcal{L}^{k}\varphi\|_{p}\leq
 \left(d\lambda^{2} \right)^{k}\|\varphi\|_{p},  \>d=\rm dim \>G,\>\>\varphi \in \mathcal{E}_{\lambda}(\mathcal{L}),\>\>1\leq p\leq \infty.
\end{equation}
One also has  \cite{Pes08} the  Bernstein-Nikolskii inequalities  
$$
\|D_{i_{1}}... D_{i_{k}}\varphi\|_{q}\leq C(\mathbf{M})
\lambda^{k+\frac{m}{p}-\frac{m}{q}}\|\varphi\|_{p}, \>k\in
\mathbb{N}, \>m=\rm dim \>\mathbf{M}, \>\varphi \in\mathcal{E}_{\lambda}(\mathcal{L}),
$$
 and 
$$
\|\mathcal{L}^{k}\varphi\|_{q}\leq C(\mathbf{M})
d^{k}\lambda^{2k+\frac{m}{p}-\frac{m}{q}}\|\varphi\|_{p}, \>k\in
\mathbb{N}, \>m=\rm dim \>\mathbf{M}, \>\varphi \in \mathcal{E}_{\lambda}(\mathcal{L}), 
$$
where $1\leq p\leq q\leq\infty$ and $C(\mathbf{M})$ is constant which depends just on the manifold.

It is known \cite{Z}, Ch. IV,  that every compact Lie group $G$ can be identified with a subgroup of orthogonal group $O(N)$ of an Euclidean space $\mathbb{R}^{N}$. It implies that every compact homogeneous manifold $\mathbf{M}$ can be identified with a submanifold which is  trajectory of a unit vector ${\bf e}\in \mathbb{R}^{N}$. Such identification of $\mathbf{M}$ with a submanifold of $\mathbb{S}^{N-1}$ is known as the  equivariant embedding  into $\mathbb{R}^{N}$. 

Having in mind the equivariant embedding of $\mathbf{M}$ into $\mathbb{R}^{N}$ one can introduce the space  $\mathbf{P}_{n}(\mathbf{M})$ of polynomials of degree $n$ on $\mathbf{M}$ as the set of restrictions to $\mathbf{M}$ of polynomials in $\mathbb{R}^{N}$ of degree $n$.  The following relations were proved in \cite{Pes08}:
$$
\textbf{P}_{n}(\mathbf{M})\subset \textbf{B}_{n}(\mathbb{D})\subset
\mathcal{E}_{n^{2}d}(\mathcal{L})\subset
\textbf{B}_{n\sqrt{d}}(\mathbb{D}), \>\>\>d=\rm dim \>G,\>\>\>n\in \mathbb{N},
$$
and
\begin{equation}\label{trig-eigen}
\bigcup _{n\in \mathbb{N}}\textbf{P}_{n}(\mathbf{M})=\bigcup_{\sigma\geq 0}
\textbf{B}_{\sigma}(\mathbb{D})=\bigcup_{j\in \mathbb{N}}
\mathcal{E}_{\lambda_{j}}(\mathcal{L}).
\end{equation}
The next Theorem  was proved in \cite{gp}, \cite{pg}.
\begin{theorem}
\label{prodthm}
If ${\bf M}=G/K$ is a compact homogeneous manifold and $\mathcal{L}$
is defined as in (\ref{Laplacian}), then for any $f$ and $g$ belonging
to $\mathcal{E}_{\omega}(\mathcal{L})$,  their pointwise product $fg$ belongs to
$\mathcal {E}_{4d\omega}(\mathcal{L})$, where $d$ is the \rm dimension of the
group $G$.

\end{theorem}

Using this Theorem and (\ref{incl}) we obtain the following

\begin{corollary}\label{prod}
 If ${\bf M}=G/K$ is a compact homogeneous manifold and $f,\> g\in  \textbf{B}_{\sigma}(\mathbb{D}) $
 then their product  $fg$ belongs to $ \textbf{B}_{2d\sigma}(\mathbb{D})$, where $d$ is the \rm dimension of the
group $G$.
 \end{corollary}

{\bf An example. Analysis on  $\mathbb{S}^{d}$}

We will specify the general setup in the case of standard unit sphere. 
Let 
$$
\mathbb{S}^{d}=\left\{x\in \mathbf{R}^{d+1}: \|x\|=1\right\}.
$$

Let $\mathcal{P}_{n}$ denote the space of spherical harmonics of degree $n$, which are restrictions to $\mathbb{S}^{d}$ of harmonic homogeneous polynomials of degree $n$ in $\mathbb{R}^{d}$. The Laplace-Beltrami operator $\Delta_{\mathbb{S}}$ on $\mathbb{S}^{d}$  is a restriction of the regular Laplace operator $\Delta$ in $\mathbb{R}^{d}$. Namely,
 $$
\Delta_{\mathbb{S}}f(x)=\Delta \widetilde{f}(x),\>\>x\in \mathcal{S}^{d},
$$
where $\widetilde{f}(x)$ is the homogeneous extension of $f$: $\>\>\widetilde{f}(x)=f\left(x/\|x\|\right)$. Another way to compute $\Delta_{\mathbb{S}}f(x)$ is to express both $\Delta_{\mathbb{S}}$ and $f$ in a spherical coordinate system.

Each $\mathcal{P}_{n}$ is the eigenspace of  $\Delta_{\mathbb{S}}$ that corresponds to the eigenvalue $-n(n+d-1)$. Let $Y_{n,l},\>\>l=1,...,l_{n}$ be an orthonormal basis in $\mathcal{P}_{n}$.

Let $e_{1},...,e_{d+1}$ be the standard orthonormal basis in $\mathbb{R}^{d+1}$.  If $SO(d+1)$ and $SO(d)$ are the groups of rotations of $\mathbb{R}^{d+1}$ and  $\mathbb{R}^{d}$ respectively then $\mathbb{S}^{d}=SO(d+1)/SO(d)$.

On $\mathbb{S}^{d}$ we consider vector fields
$$
X_{i,j}=x_{j}\partial_{x_{i}}-x_{i}\partial_{x_{j}}
$$
which are generators of one-parameter groups of rotations   $\exp tX_{i,j}\in SO(d+1)$ in the plane $(x_{i}, x_{j})$. These groups are defined by the formulas 
$$
\exp \tau X_{i,j}\cdot (x_{1},...,x_{d+1})=(x_{1},...,x_{i}\cos \tau -x_{j}\sin \tau ,..., x_{i}\sin \tau +x_{j}\cos \tau ,..., x_{d+1})
$$
Let $e^{\tau X_{i,j}}$ be a one-parameter group which is a representation of $\exp \tau X_{i,j}$ in a space $L_{p}(\mathbb{S}^{d})$. It acts on $f\in L_{p}(\mathbb{S}^{d})$ by the following formula
$$
e^{\tau X_{i,j}}f(x_{1},...,x_{d+1})=f(x_{1},...,x_{i}\cos \tau -x_{j}\sin \tau ,..., x_{i}\sin \tau +x_{j}\cos \tau ,..., x_{d+1}).
$$
Let $D_{i,j}$ be a generator of $e^{\tau X_{i,j}}$ in $L_{p}(\mathbb{S}^{d})$. In a standard way the Laplace-Beltrami operator $\mathcal{L}$ can be identified with an operator in $L_{p}(\mathbb{S}^{d})$ for which we will keep the same notation. One has 
 $$
\Delta_{\mathbb{S}}=\mathcal{L}=\sum_{(i,j)}D_{i,j}^{2}.
$$

\section{Applications}\label{App}

\subsection{Compact homogeneous manifolds}\label{compact}

We return to  setup of subsection 5.1.
Since $D_{j},\>1\leq j\leq d$ generates a group $e^{\tau D_{j}}$ in $L_{p}(\mathbf{M})$ the formulas 
 (\ref{B1}) give for $f\in \mathbf{B}_{\sigma}(\mathbb{D})$

\begin{equation}\label{B11}
D_{j}^{2m-1}f(x)=\mathcal{B}^{(2m-1)}_{j}(\sigma)f(x)=
\left(\frac{\sigma}{\pi}\right)^{2m-1}\sum_{k\in \mathbb{Z}}(-1)^{k+1}A_{m,k}e^{\frac{\pi}{\sigma}(k-1/2)D_{j}}f(x),\>\>m\in \mathbb{N},
\end{equation}

\begin{equation}\label{B22}
D_{j}^{2m}f=\mathcal{B}^{(2m)}_{j}(\sigma)f=
\left(\frac{\sigma}{\pi}\right)^{2m}\sum_{k\in \mathbb{Z}}(-1)^{k+1}B_{m,k}e^{\frac{\pi k}{\sigma}D_{j}}f,\>\>m\in \mathbb{N}\cup{0}.
\end{equation}

Note, that  every  vector field $X$ on $\mathbf{M}$ is a linear combination  $\sum_{j=1}^{d}a_{j}(x)D_{j},\>\>\>x\in \mathbf{M}$. Thus we can formulate the following fact.

\begin{theorem}\label{lincomb}
If $f\in \mathbf{B}_{\sigma}(\mathbb{D})$ then for every vector field  $X=\sum_{j=1}^{d}a_{j}(x)D_{j}$ on $\mathbf{M}$
\begin{equation}
Xf=\sum_{j=1}^{d}a_{j}(x)\mathcal{B}_{j}(\sigma)f,
\end{equation}
where $\mathcal{B}_{j}(\sigma)=\mathcal{B}_{D_{j}}^{1}(\sigma)$.
\end{theorem}

Moreover, every linear combination $X=\sum_{j=1}^{d}a_{j}X_{j}$ with constant coefficients  can be identified with a generator $D=\sum_{j=1}^{d}a_{j}D_{j}$ of a bounded strongly continuous group of operators $e^{tD}$ in $L_{p}(\mathbf{M}),\>\>\>1\leq p\leq \infty$.

A commutator 
\begin{equation}\label{com}
[D_{l}, D_{m}]=D_{l}D_{m}-D_{m}D_{l}=\sum_{j=1}^{d}c_{j}D_{j}
\end{equation}
 where  constant  coefficients $c_{j}$ here  are known as  structural constants of the Lie algebra  is another generator of an one-parameter group of translations.
Formulas (\ref{B11})   imply the following relations.

\begin{theorem}
If $D=\sum_{j=1}^{d}a_{j}D_{j}$ then for operators   $\mathcal{B}_{D}(\sigma) =\mathcal{B}_{D}^{1}(\sigma)  $  and all  $f\in \mathbf{B}_{\sigma}(\mathbb{D})$
\begin{equation}
\mathcal{B}_{D}(\sigma)f=\sum_{j=1}^{d}a_{j}\mathcal{B}_{j}(\sigma)f.
\end{equation}
In particular 
\begin{equation}
[D_{l}, D_{m}]f=\mathcal{B}_{l}(\sigma)\mathcal{B}_{m}(\sigma)f-\mathcal{B}_{m}(\sigma)\mathcal{B}_{l}(\sigma)f=\sum_{j=1}^{d}c_{j}\mathcal{B}_{j}(\sigma)f.
\end{equation}

Moreover,
\begin{equation}
D_{j_{1}}...D_{j_{k}}=\mathcal{B}_{j_{1}}(\sigma)...\mathcal{B}_{j_{k}}(\sigma)f.
\end{equation}
\end{theorem}
Clearly, for any two smooth functions $f,\>g$ on $\mathbf{M}$ one has 
$$
D_{j}(fg)(x)=f(x)D_{j}g(x)+g(x)D_{j}f(x).
$$
If $D=\sum_{j=1}^{d}a_{j}(x)D_{j}$ then  for $f,  \>g\in \mathbf{B}_{\sigma}(\mathbb{D})$ the following equality holds
$$
D\left(fg\right)(x)=\sum_{j=1}^{d}a_{j}(x)\left\{f(x)\mathcal{B}_{j}(\sigma)g(x)+g(x)\mathcal{B}_{j}(\sigma)f(x)\right\}.
$$
It is the Corollary \ref{prod} which allows to formulate the following result.
\begin{theorem}
If $f,  \>g\in \mathbf{B}_{\sigma}(\mathbb{D})$ and $D=\sum_{j=1}^{d}a_{j}D_{j}$, where $a_{j}$ are constants then 
$$
\mathcal{B}_{D}(2d\sigma)(fg)(x)=\sum_{j=1}^{d}a_{j}\left\{f(x)\mathcal{B}_{j}(\sigma)g(x)+g(x)\mathcal{B}_{j}(\sigma)f(x)\right\}.
$$

\end{theorem}

\bigskip

The formula $-\mathcal{L}=D_{1}^{2}+ D_{2}^{2}+ ...+ D_{d}^{2}$ implies the following result.

\begin{theorem}
If $f\in \mathbf{B}_{\sigma}(\mathbb{D})$ then
\begin{equation}\label{lap}
\mathcal{L}f=\mathcal{B}_{\mathcal{L}}^{2}(\sigma)f=\sum_{j=1}^{d}\mathcal{B}_{j}^{2}(\sigma)f.
\end{equation}

\end{theorem}

\begin{remark}
Note that it is not easy to find "closed" formulas for groups like $e^{it\mathcal{L}}$. Of course, one always has  a representation 
\begin{equation}\label{exp}
e^{it\mathcal{L}}f(x)=\int_{\mathbf{M}}K(x,y)f(y)dy, 
\end{equation}
with  $K(x,y)=\sum_{\lambda_{j}}e^{it\lambda_{j}}u_{j}(x)\overline{u}_{j}(y)$, where  $\{u_{j}\}$ is a complete orthonormal system of eigenfunctions of $\mathcal{L}$ in $L_{2}(\mathbf{M})$ and $\mathcal{L}u_{j}=\lambda_{j}$. But the formula (\ref{exp}) doesn't tell much about $e^{it\mathcal{L}}$. In other words the explicit formulas for  operators $\mathcal{B}_{\mathcal{L}}^{2}(\sigma)$  usually  unknown. At the same time it is easy to understand the right-hand side in (\ref{lap})   since  it is  coming from translations on a manifold in certain basic directions (think for example of a sphere).

\end{remark}

Let us introduce the  notations
$$\mathcal{B}^{(2m-1)}_{j}(\sigma,N)f=
\left(\frac{\sigma}{\pi}\right)^{2m-1}\sum_{|k|\leq N}(-1)^{k+1}A_{m,k}e^{\frac{\pi}{\sigma}(k-1/2)D_{j}}f,
$$
$$
\mathcal{B}^{(2m)}_{j}(\sigma, N)f=
\left(\frac{\sigma}{\pi}\right)^{2m}\sum_{|k|\leq N}(-1)^{k+1}B_{m,k}e^{\frac{\pi k}{\sigma}D_{j}}f.
$$
The following approximate Boas-type formulas hold. 
\begin{theorem}
If $f\in \mathbf{B}_{\sigma}(\mathbb{D})$ then

\begin{equation}
\sum_{j=1}^{d}\alpha_{j}(x)D_{j}f=\sum_{j=1}^{d}\alpha_{j}(x)\mathcal{B}_{j}(\sigma, N)f+O(N^{-2}),
\end{equation}

and if $a_{j}$ are constants then
\begin{equation}
\mathcal{B}_{D}(\sigma)f=\sum_{j=1}^{d}a_{j}\mathcal{B}_{j}(\sigma, N)f+O(N^{-2}),\>\>D=\sum_{j=1}^{d}a_{j}D_{j}.
\end{equation}
Moreover,

\begin{equation}
[D_{l}, D_{m}]f=\sum_{j=1}^{d}c_{j}\mathcal{B}_{j}(\sigma,N)f+O(N^{-2}),
\end{equation}
\begin{equation}
\mathcal{L}f=\sum_{j=1}^{d}\mathcal{B}_{j}^{2}(\sigma, N)f+O(N^{-2}).
\end{equation}

\end{theorem}

\subsection{ The Heisenberg group.}\label{non-compact}

In the space $\mathbb{R}^{2n+1}$ with coordinates $(x_{1},...,x_{n},y_{1},...,y_{n}, t)$ we consider vector fields 
$$
X_{j}=\partial_{x_{j}}-\frac{1}{2}y_{j}\partial_{t},\>\>\>Y_{j}=\partial_{y_{j}}+\frac{1}{2}x_{j}\partial_{t},\>\>\> T=\partial_{t}, \>\>1\leq j\leq n.
$$
As operators in the regular space $L_{p}(\mathbb{R}^{2n+1}),\>\>1\leq p\leq \infty,$ they generate one-parameter bounded strongly continuous groups of operators. In fact, these operators form in $L_{p}(\mathbb{R}^{2n+1})$ a representation  of the Lie algebra of the so-called Heisenberg group $\mathbb{H}^{n}$.
The corresponding one-parameter groups are 
$$
e^{\tau X_{j}}f(x_{1},..., t)=f(x_{1},..., x_{j}+\tau,.....,x_{n}, y_{1},...,y_{n}, t-\frac{1}{2}y_{j}\tau)
$$

$$
e^{\tau Y_{j}}f(x_{1},..., t)=f(x_{1},..., x_{n}, y_{1},..., y_{j}+\tau,..., y_{n}, t+\frac{1}{2}x_{j}\tau).
$$
As we already know for every $\sigma>0$ there exists a non-empty  set $\mathbb{B}_{\sigma}(X_{j})$ such that their union is dense in $L_{p}(\mathbb{R}^{2n+1}),\>\>1\leq p\leq \infty,$ and for which the following formulas  hold with $m\in \mathbb{N}$
$$
X^{2m-1}_{j}f=
\left(\frac{\sigma}{\pi}\right)^{2m-1}\sum_{k\in \mathbb{Z}}(-1)^{k+1}A_{m,k}e^{\frac{\pi}{\sigma}(k-1/2)X_{j}}f,
$$

$$
X^{2m-1}_{j}f=
\left(\frac{\sigma}{\pi}\right)^{2m-1}\sum_{|k|<N}(-1)^{k+1}A_{m,k}e^{\frac{\pi}{\sigma}(k-1/2)X_{j}}f+O(N^{-2}),
$$
 for every $\>f\in \mathbb{B}_{\sigma}(X_{j})$. One can easily obtain a number  of similar formulas for $Y_{j}$ and $T$.

\subsection{The Schr\"{o}dinger representation}\label{Schrod}

 Take  $3$-\rm dimensional Heisenberg group and consider what is known as its  Schr\"{o}dinger representation in the regular space $L_{2}(\mathbb{R})$ \cite{Folland}.
The infinitesimal operators of this representation are differentiaton $D=\frac{1}{2\pi i}\frac{d}{dx}$ and multiplication by independent variable $x$ which will be denoted as $X$. Every linear combination $pD+qX$ with constant coefficients $p,\>q$ generates a unitary group in $L_{2}(\mathbb{R})$ according to the formula 
\begin{equation}\label{LC}
e^{2\pi i(pD+qX)}f(x)=e^{2\pi i q x+\pi i pq}f(x+p),
\end{equation}
and in particular
\begin{equation}\label{X}
e^{2\pi i qX}f(x)=e^{2\pi i q x}f(x),\>\>\>\>\>
e^{2\pi i pD}f(x)=f(x+p).
\end{equation}
For every $\sigma>0$ one can consider corresponding spaces $\mathbf{B}_{\sigma}(pD+qX), \mathbf{B}_{\sigma}(D),\mathbf{B}_{\sigma}(X)$ and corresponding operators $\mathcal{B}$ and $\mathcal{Q}$. One of possible Boas-type  formulas would look like
$$
(pD+qX)^{m}f(x)=\mathcal{B}_{pD+qX}^{m}(\sigma)f(x)
$$
and holds for $f\in \mathbf{B}_{\sigma}(pD+qX)$ where $\bigcup_{\sigma\geq 0} \mathbf{B}_{\sigma}(pD+qX)$ is dense in $L_{2}(\mathbb{R})$.

\bigskip

\begin{remark} As it was noticed in \cite{Pes88} the intersection of $\mathbf{B}_{\sigma}(D)$ and $\mathbf{B}_{\sigma}(X)$ contains only $0$. It follows from the fact that $\mathbf{B}_{\sigma}(D)$ is the regular Paley-Wiener space and $\mathbf{B}_{\sigma}(X)$ is the space of functions whose support is in $[-\sigma,\>\sigma], \>\sigma>0$. As a result a formula like 
$$
(pD+qX)f(x)=p\mathcal{B}_{D}(\sigma)f(x)+q\mathcal{B}_{X}(\sigma)f(x)
$$
holds only for $f= 0$ (unlike similar formulas in Theorem \ref{lincomb}).
\end{remark}
\begin{remark}
We note that the operator $D^{2}+X^{2}$ is self-adjoint and  the set $\bigcup_{\sigma\geq 0}B_{\sigma}(D^{2}+X^{2})$ is a span of all Hermit functions.
\end{remark}

\end{document}